\documentclass[11pt]{amsart}
\usepackage{amsmath,amssymb,amsthm,upref,graphicx,mathrsfs}

\usepackage{color,xcolor}
\usepackage[
  colorlinks=true,
  linkcolor=blue,
  citecolor=blue,
  urlcolor=blue,
  ]{hyperref}

\numberwithin{equation}{section}
\allowdisplaybreaks

\newtheorem{theorem}{Theorem}[section]

\newtheorem{lemma}[theorem]{Lemma}

\theoremstyle{definition}

\newcommand{\Rn}{\mathbb{R}^n}

\newcommand{\supp}{{\rm supp}\,}

\def\XXint#1#2#3{{\setbox0=\hbox{$#1{#2#3}{\int}$}
\vcenter{\hbox{$#2#3$}}\kern-.5\wd0}}

\let \la=\lambda

\begin{document}

\title[On General multilinear square function with non-smooth kernels]
  {On General multilinear square function with non-smooth kernels}

\authors

\author[M. Hormozi]{Mahdi Hormozi}
\address{Mahdi Hormozi \\
Department of Mathematical Sciences, Division of Mathematics,
University of Gothenburg, Gothenburg 41296, Sweden}
\email{hormozi@chalmers.se}

\author{Zengyan Si}
\address{
        Zengyan Si\\
       School of Mathematics and Information Science\\
       Henan Polytechnic University\\ Jiaozuo 454000\\People's
Republic of China} \email{zengyan@hpu.edu.cn}
\thanks{The second author was supported by the National Natural Science Foundation of China(No.11401175) and Doctor Foundation of Henan Polytechnic University(No.B2012-055).The third author was supported partly by NSFC
(No.11471041), the Fundamental Research Funds for the Central Universities (No.2012CXQT09 and No.2014kJJCA10) and NCET-13-0065.  }

\author{Qingying Xue}
\address{Qingying Xue
\\
School of Mathematical Sciences
\\
Beijing Normal University
\\
Laboratory of Mathematics and Complex Systems
\\
Ministry of Education
\\
Beijing 100875
\\
People's Republic of China } \email{qyxue@bnu.edu.cn}


\arraycolsep=1pt

%
%
\begin{abstract}
In this paper, we obtain some boundedness of multilinear square functions $T$ with non-smooth kernels, which extend some known results significantly. The corresponding multilinear maximal square function $T^*$ was also introduced and weighted strong and weak
type estimates for $T^*$ were given.
\end{abstract}

\maketitle


\section{Introduction}
\label{Sect:Introduction}

\par


%

It is well-known that the multilinear Calder$\acute{o}$n-Zygmund operators were introduced
and first studied by Coifman and Meyer \cite{coif1, coif2, coif3}, and later by Grafakos and Torres \cite{Gra1, Gra2}. The study
of this subject was recently enjoyed a resurgence of renewed
interest and activity. In particular, the study of multilinear
singular integral operators with non-standard kernels have recently
received pretty much attention. 

Before we state some known results, we begin by giving some definitions and notations.
For any $v\in(0,\infty),$ a locally integrable function $K_v(x,y_1,\dots, y_m)$  defined away from the diagonal
 $x=y_1=\cdots =y_m$ in $(\mathbb{R}^n)^{m+1}$ satisfies the integral condition of $C-Z$ type, if there are some positive constants $\gamma, A,$  and $B>1,$  such that
 \begin{equation}\label{size1}
\begin{split}
\left( \int_{0}^\infty |K_v(x,y_1,\cdots,y_m)|^2 \frac{dv}{v}\right)^{\frac 12}\leq \frac{A}{(\sum_{j=1}^m|x-y_j|)^{mn}},
\end{split}
\end{equation}

\begin{equation}\label{eq2}
 \begin{split}
  &\left( \int_{0}^\infty |K_v(z,y_1,\cdots,y_m)-K_v(x,y_1,\cdots,y_m)|^2 \frac{dv}{v}\right)^{\frac 12}
  \leq  \frac{A|z-x|^\gamma}{(\sum_{j=1}^m|x-y_j|)^{mn+\gamma}},
\end{split}
\end{equation}
whenever $|z-x|\leq \frac{1}{B} \max_{j=1}^m{|x-y_j|}$; and
\begin{equation}\label{eq3}
\begin{split}
 &\left( \int_{0}^\infty |K_v(x,y_1,\dots,y_i,\dots,y_m)-K_v(x,y_1,\dots,y'_i,\dots,y_m)|^2 \frac{dv}{v}\right)^{\frac 12}
  \leq \frac{A |y_j-y_{j}'|^\gamma}{(\sum_{j=1}^m|x-y_j|)^{mn+\gamma}}
\end{split}
\end{equation}
whenever $|y-y'|\leq \frac{1}{B} \max_{j=1}^m{|x-y_j|}$.
\par
Xue and Yan \cite{XY} introduced a multilinear square function $T$ which is defined by
 \begin{equation}
\begin{split}
T(\vec{f})(x)=\left( \int_{0}^\infty \left|\int_{(\mathbb{R}^n)^m}K_v(x,y_1,\dots,y_m) \prod_{j=1}^mf_{j}(y_j)dy_1,\dots,dy_m\right|^2\frac{dv}{v}\right)^{\frac 12},
\end{split}
\end{equation}
for any $\vec{f}=(f_1,\dots,f_m)\in  \mathcal{S}(\mathbb{R}^n)\times \cdots \times \mathcal{S}(\mathbb{R}^n)$ and all $x\notin \bigcap_{j=1}^m \texttt{supp} f_j.$  We assume that for some $1\leq q_1,\cdots,q_m
< \infty$ and $0<q<\infty,$ $T$  can be extended to a bounded multilinear operator from
$L^{q_1}\times \cdots \times L^{q_m}$ to $L^q$, where
$\frac{1}{q}=\frac{1}{q_1}+\cdots +\frac{1}{q_m}.$
\par Xue and Yan \cite{XY} obtain the following results.
\par
\textbf{Theorem A}  (\cite{XY})
  Let $T$ be a multilinear square function with the kernel satisfying the integral condition of $C-Z$ type.
  Then $T$ can be extended to a bounded operator from $L^1(\Rn)\times\ldots\times L^1(\Rn)$ to $L^{1/m,\infty}(\Rn)$.

\par
\textbf{Theorem B}  (\cite{XY})
 Let $T$ be a multilinear square function with the kernel satisfying the integral condition of $C-Z$ type.  Let $\frac{1}{p}=\frac{1}{p_1}+\ldots+\frac{1}{p_m}$ with $1\leq p_1,\ldots,p_{m}< \infty$ and assume that $\vec{\omega}$ satisfies $A_{\vec{p}}$ condition, then the following results hold:

 \begin{enumerate}
   \item If there is no $p_i=1$, then $\|T(\vec{f}) \|_{L^{p}(\nu_{\vec{w}})}\leq C \prod_{i=1}^{m} \|f_i\|_{L^{p_i}(w_i)} $\\
   \item If there is a $p_i=1$, then $\|T(\vec{f}) \|_{L^{p,\infty}(\nu_{\vec{w}})}\leq C \prod_{i=1}^{m} \|f_i\|_{L^{p_i}(w_i)} $
 \end{enumerate}

Recently, many mathematicians are concerned to remove or replace the smoothness
condition on the kernel  \cite{anh, duong, ML, L, duonggong, MN, LZ, YXY}.
It is natural to ask under the non-smooth condition, does Theorem A and Theorem B still hold or not ? In this paper we can give a positive answer: we can extend Theorem A and Theorem B to non-smooth case. Moreover, we  introduce the multilinear maximal square function $T^*$ and  weighted  strong and weak
type estimates  for $T^*$ are  also given.
\par To begin with, we first recall  a class of integral operators
$\{A_t\}_{t>0}$, that plays the role of an approximation
to the identity modifying the original definition in \cite{DM}
to extend it to a more general scenario. We assume that the operators $A_{t}$ are associated with kernels $a_{t}(x,y)$ in the sense that\\
$$A_{t}f(x)=\int_{\mathbb{R}^{n}}a_{t}(x,y)f(y)dy$$ for every function $f\in L^{p}(\mathbb{R}^{n})$, $1\leq p\leq \infty$, and the kernels $a_{t}(x,y)$ satisfy the following size conditions\\
\begin{equation}\label{sizecondition1}
|a_{t}(x,y)|\leq h_{t}(x,y):=t^{-n/s}h\left(\frac{|x-y|}{t^{1/s}}\right),
\end{equation}
where $s$ is a positive fixed constant and $h$ is a positive, bounded, decreasing function satisfying\\
\begin{equation}\label{smallcondition2}
\lim_{r\rightarrow0}r^{n+\eta}h(r^{s})=0
\end{equation}
for some $\eta>0$. These conditions imply that for some $C>0$ and all $0<\eta\leq\eta^{\prime}$, the kernels $a_{t}(x,y)$ satisfy $$|a_{t}(x,y)|\leq C t^{-n/s}(1+t^{-1/s}|x-y|)^{-n-\eta^{\prime}}.$$

\textbf{Assumption(H1)} Assume that for each $i=1,\cdots,m$ there exist operators $\{A_t^{(i)}\}_{t>0}$ with
kernels $a^{i}_t(x,y)$ that satisfy condition \eqref{sizecondition1} and \eqref{smallcondition2} with constants $s$ and $\eta$
and that for every $j=0,1,2,\cdots,m, $ there
exist kernels $K_{t,v}^{(i)}$ such that

\begin{equation}\label{H1}
\begin{split}
&<T(f_1,\cdots, A_t^{(i)}f_i,\cdots,f_m),g>\\
&=\int_{\mathbb{R}^n}\left(\int_0^\infty\left|\int_{(\mathbb{R}^n)^m}K_{t,v}^{(i)}(x,y_1,\cdots,y_m)f_1(y_1)\cdots f_m(y_m)d\vec{y}\right|^2\frac{dv}{v}\right)^{1/ 2}g(x)dx,
\end{split}
\end{equation}
for all $f_1,\cdots,f_m,g$ in $\mathcal{S}$ with $\cap_{k=1}^m
\texttt{supp} f_k \cap \texttt{supp}\, g=\phi.$
There exists a
function $\phi\in C(\mathbb{R})$ with $\texttt{supp}\, \phi\in [-1,1] $ and a
constant $\varepsilon>0$ so that for every $j=0,1,\cdots,m$ and
every $i=1,2,\cdots,m,$ we have

\begin{equation}\label{H2}
\begin{split}
& \left(\int_0^\infty\left|K_v(x,y_1,\cdots,y_m)-K_{t,v}^{(i)}(x,y_1,\cdots,y_m)\right|^2\frac{dv}{v}\right)^{1/ 2}\\
&\leq \frac{A}{(|x-y_1|+\cdots +|x-y_m|)^{mn}}\sum_{k=1,k\neq i}^m
\phi(\frac{|y_i-y_k|}{t^{1/s}})+\frac{At^{\varepsilon/s}}{(|x-y_1|+\cdots
+|x-y_m|)^{mn+\varepsilon}}
\end{split}
\end{equation}
whenever $t^{1/s}\leq |x-y_i|/2.$
\par

\textbf{Assumption (H2)}\indent Assume that there exist operators $\{A_{t}\}_{t>0}$ with kernels $a_{t}(x,y)$ that satisfy condition \eqref{sizecondition1} and \eqref{smallcondition2} with constants $s$ and $\eta$, and there exist kernels $K_{t,v}^{(0)}(x,y_{1},\cdots,y_{m})$ such that for all $x,y_1,\cdots,y_m \in \mathbb{R}^{n}$ and $t>0$ the representation is valid\\
\begin{equation}\label{sizeforkernel}
\aligned
K_{t,v}^{(0)}(x,y_{1},\cdots,y_{m})=\int_{\mathbb{R}^{n}}K_v(z,y_{1},\cdots,y_{m})a_{t}(x,z)dz.
\endaligned
\end{equation}
Assume also that there exist a function $\phi \in \mathcal{C}(\mathbb{R})$ and $\phi \subset [-1,1]$ and a constant $\varepsilon>0$ such that
 \begin{equation}\label{sizeforkernel2}
\aligned
\left(\int_0^\infty \left|K_{t,v}^{(0)}(x,y_{1},\cdots,y_{m})\right|^2\frac{dv}{v}\right)^{1/ 2}\leq \frac{A}{(\sum_{k=1}^{m}|x-y_{k}|)^{mn}},
\endaligned
\end{equation}
whenever $2t^{1/s}\leq \min_{1\leq j \leq m}|x-y_{j}|$ and
 \begin{equation}\label{smoothnessforkernel1}
\aligned
&\left(\int_0^\infty \left|K_v(x,y_{1},\cdots,y_{m})-K_{t,v}^{(0)}(x,y_{1},\cdots,y_{m})\right|^2\frac{dv}{v}\right)^{1/ 2}\\
&\leq \frac{A}{(\sum_{k=1}^{m}|x-y_{k}|)^{mn}}\sum_{\stackrel{k=1}{k\neq j}}^{m}\phi\left(\frac{|x-y_{k}|}{t^{1/s}}\right)+\frac{At^{\varepsilon/s}}{(\sum_{k=1}^{m}|x-y_{k}|)^{mn+\varepsilon}}
\endaligned
\end{equation}
for some $A>0$, whenever $2t^{1/s}\leq \max_{1\leq j \leq m}|x-y_{j}|$.
\\

\textbf{Assumption (H3)}\indent Assume that there exist operators $\{A_{t}\}_{t>0}$ with kernels $a_{t}(x,y)$ that satisfy condition \eqref{sizecondition1} and \eqref{smallcondition2} with constants $s$ and $\eta$ and there exist kernels $K_{t,v}^{(0)}$ such that \eqref{sizeforkernel} holds. Also assume that there exist positive constant $A$ and $\varepsilon$ such that,
 \begin{equation}\label{smoothnessforkernel2}
\aligned
\left(\int_0^\infty \left| K_{t,v}^{(0)}(x,y_{1},\cdots,y_{m})-K_{t,v}^{(0)}(x^{\prime},y_{1},\cdots,y_{m})\right|^2\frac{dv}{v}\right)^{1/ 2}\leq\frac{At^{\varepsilon/s}}{(\sum_{k=1}^{m}|x-y_{k}|)^{mn+\varepsilon}},
\endaligned
\end{equation}
whenever $2t^{1/s}\leq \min_{1\leq j \leq m}|x-y_{j}|$ and $2|x-x^{\prime}|\leq t^{1/s}$.

%

Kernels $K_v$ that satisfying \eqref{size1},\eqref{H1} and \eqref{H2} with parameters $m,A,s,\eta,\varepsilon$ are called generalized square function kernels, and their collection is denoted by $ m-GSFK(A,s,\eta,\varepsilon)$. We say that $T$ is of class $ m-GSFO(A,s,\eta,\varepsilon)$ if T has an associated kernel $K_v$ in $ m-GSFK(A,s,\eta,\varepsilon)$.
%
%
%


\section{Endpoint estimate for $T$}
\label{Sect:Unweighted}
%

%
In this section we prove the endpoint estimate for multilinear generalized square function.
 \begin{theorem}\label{thm:endpoint}
  Let $T$ be a multilinear operator in $m-GSFO(A,s,\eta,\varepsilon)$.
  Then $T$ can be extended to a bounded operator from $L^1(\Rn)\times\ldots\times L^1(\Rn)$ into $L^{1/m,\infty}(\Rn)$ with bound
  \begin{equation*}
    ||T||_{L^1\times\ldots\times L^1 \to L^{1/m,\infty}}\leq C_{n,m} (A+||T||_{L^{p_1}\times \dots \times L^{p_m}\to L^{p,\infty}}).
  \end{equation*}
   \end{theorem}
\begin{proof}[Proof of Theorem \ref{thm:endpoint}]

For simplicity, we write $\|T\|=\|T\|_{{L^{p_1}\times \dots \times L^{p_m}\to L^{p,\infty}}}$.
  Let $\vec f=(f_1,\dots, f_m)$.
    By homogeneity,  we may assume that each $f_i\in L^1(\Rn)$ and
    $\|f_i\|_{L^1(\Rn)}=1$. It suffices to prove that for any $\la>0$,
  \begin{eqnarray}
  |\{x\in\Rn:\, |T(\vec f\,)(x)|>2^m\la\}| \le C (A+ \|T\|)^{1/m}\,\la^{-1/m}.
  \end{eqnarray}

For  $j=1,2,\dots,m$, we consider the Calder\'on--Zygmund decomposition of each function $f_j$
at level $(\alpha\la)^{1/m}$, where $\alpha$ is  a positive constant to be determined later. Then, each $f_j$ has the decomposition
$$f_j=g_j+b_j=g_j+\sum_{k\in I_j} b_{j,k}$$
with $I_j$ being some index set, such that
\begin{enumerate}
\item[(i)] $|g_j(x)|\le C (\alpha\la)^{1/m}$ for all $x\in\Rn$;
\item[(ii)] there exists a sequence of mutually disjoint cubes $\{Q_{j,k}\}_{k\in I_j}$ such that $\supp b_{j,k}\subset Q_{j,k}$, $\int_{\Rn} b_{j,k}(x)\,dx=0$ and
    $\|b_{j,k}\|_{L^1(\Rn)}\le C (\alpha\la)^{1/m}\,|Q_{j,k}|;$
\item[(iii)] $\sum_{k\in I_j} |Q_{j,k}| \le C (\alpha\la)^{-1/m}$.
\end{enumerate}
It follows from (ii) and (iii) that
$$\|b_j\|_{L^1(\Rn)}\le C$$
and hence
\begin{eqnarray}\label{eq:xx1}
\|g_j\|_{L^1(\Rn)}\le \|f_j\|_{L^1(\Rn)}+\|b_j\|_{L^1(\Rn)}\le 1+C.
\end{eqnarray}
Further,
\begin{eqnarray}\label{eq:e1}
\|g_j\|_{L^{p_j}(\Rn)}\le (1+C)^{1/{p_j}}[C (\alpha\la)^{1/m}]^{(p_j-1)/p_j}
= C (\alpha\la)^{1/(m p_j')}.
\end{eqnarray}
Let
\begin{eqnarray*}
E_\la^{(1)}&&:= \{x\in\Rn:\, |T(g_1,g_2,\dots, g_m)(x)|>\la\};\\
E_\la^{(2)}&&:= \{x\in\Rn:\, |T(b_1,g_2,\dots, g_m)(x)|>\la\};\\
E_\la^{(3)}&&:= \{x\in\Rn:\, |T(g_1,b_2, g_3, \dots, g_m)(x)|>\la\};\\
&&\vdots\\
E_\la^{(2^m)}&&:= \{x\in\Rn:\, |T(b_1,b_2,\dots, b_m)(x)|>\la\}.\\
\end{eqnarray*}
That is, for $1\le i\le 2^m$,
$$
E_\la^{(i)}= \{x\in\Rn:\, |T(h_1,h_2,\dots, h_m)(x)|>\la\},\qquad h_j\in\{g_j,\, b_j\}
$$
and all the sets $E_\la^{(s)}$ are distinct. Thus,
\begin{eqnarray*}
 |\{x\in\Rn:\, |T(\vec f\,)(x)|>2^m\la\}| \le  \sum_{i=1}^{2^m} |E_\la^{(i)}|,
\end{eqnarray*}
and we only need to prove that for $1\le i\le 2^m$ and $\la>0$,
\begin{eqnarray}\label{eq:e2}
|E_\la^{(i)}|\le C (A+ \|T\|)^{1/m}\,\la^{-1/m}.
\end{eqnarray}

First, we estimate $|E_\la^{(1)}|$. By the $ L^{p_1}(\Rn)\times \dots \times L^{p_m}(\Rn)\to L^{p,\infty}(\Rn)$ boundedness of $T$ and Chebychev's inequality, one gets that
\begin{eqnarray}\label{eq:e3}
|E_\la^{(1)}|
&&\le \left(\frac{\|T\|\, \prod_{j=1}^m\|g_j\|_{L^{p_j}(\Rn)}}{\la}\right)^p\\
&&\le C \left(\frac{\|T\|\, (\alpha\la)^{\frac1m \sum_{j=1}^m\frac1{ p_j'}}}{\la}\right)^p = C \|T\|^p\alpha^{p-\frac1m} \,\la^{-\frac 1m}. \notag
\end{eqnarray}

Now  we estimate $|E_\la^{(i)}|$ for $2\le i\le 2^m$. Suppose that for some $1\le \ell \le m$ we have $\ell$ bad functions and $m-\ell$ good functions appearing in $T(h_1, \dots, h_m)$, where $h_j\in\{g_j, b_j\}$.
It suffices to prove that
\begin{eqnarray}\label{eq:e4}
|E_\la^{(i)}| \le C\,\la^{-\frac 1m}\,\left[\|T\|^p \alpha^{p-\frac1m}  + \alpha^{-\frac 1m}+ A \alpha^{1-\frac 1m}\right].
\end{eqnarray}
Indeed, once we have \eqref{eq:e4}, then combining with \eqref{eq:e3} and selecting $\alpha=(\|T\|+A)^{-1}$, we get \eqref{eq:e2}.

Now we prove \eqref{eq:e4}. For each cube $Q_{j,k}$ obtained in the Calder\'on-Zygmund decomposition in (ii),
we denote by $Q_{j,k}^\ast$ the cube with the same center as $Q_{j,k}$ but
$5\sqrt{n}$ times the side length of that of $Q_{j,k}$. By (iii), we have
\begin{eqnarray*}
\left|\bigcup_{j=1}^m \bigcup_{k\in I_j} Q_{j,k}^\ast\right|
\le \sum_{j=1}^m\,\sum_{k\in I_j} |Q_{j,k}^\ast|
\le Cm  (5\sqrt{n})^n(\alpha\la)^{-1/m}.
\end{eqnarray*}
Thus, to get \eqref{eq:e4}, it suffices to prove that
\begin{eqnarray}\label{eq:e5}
&&\left|\left\{x\notin \bigcup_{j=1}^m \bigcup_{k\in I_j} Q_{j,k}^\ast:\, |T(h_1,\dots, h_m)(x)|>\la\right\}\right|\\
&&\qquad\le C\,\la^{-\frac 1m}\,\left[\|T\|^p \alpha^{p-\frac1m}  + \alpha^{-\frac 1m}+ A \alpha^{1-\frac 1m}\right], \qquad h_j\in\{g_j, b_j\}. \notag
\end{eqnarray}
Without loss of generality, we may assume that
$$h_j=\begin{cases}
b_j,\qquad & j=1,\dots,\ell;\\
g_j,\qquad &j=\ell+1,\dots, m.
\end{cases}$$
 Fix  $x\notin \bigcup_{j=1}^m \bigcup_{k\in I_j} Q_{j,k}^\ast$.
Then, by the Calder\'on-Zygmund decomposition, we write
\begin{eqnarray}\label{eq:e6}
&&T(b_1,\dots, b_\ell, g_{\ell+1},\dots, g_m)(x)\\
&&\quad= \sum_{k_1,k_2,\dots, k_\ell} T(b_{1, k_1}\,,\dots, b_{\ell, k_\ell},\, g_{\ell+1},\dots, g_m)(x)\notag\\
&&\quad= \sum_{j=1}^\ell \left(\sum_{\substack{k_1,\dots,k_{j-1} \\ k_{j+1},\dots, k_\ell}}\;\sum_{k_j\in \Theta_j} T(b_{1, k_1}\,,\dots, b_{\ell, k_\ell},\, g_{\ell+1},\dots, g_m)(x)\right) \notag\\
&&\quad=:\sum_{j=1}^\ell T^{(j)}(b_1,\dots, b_\ell, g_{\ell+1},\dots, g_m)(x), \notag
\end{eqnarray}
where, for $1\le j\le\ell$, we recall that  $\supp b_{j,k_j}\subset Q_{j,k_j}$ and define
\begin{equation*}\begin{split}
\Theta_j:=\{k_j\in I_j:\,
&\ell(Q_{i, k_i}) >\ell(Q_{j,k_j})\; \textup{for}\;1\le i\le j-1, \\ &
\ell(Q_{j,k_j})\le \ell(Q_{i, k_i})\; \textup{for}\;j+1\le i\le \ell\}.
\end{split}\end{equation*}
Notice that $\Theta_j$ is given by collecting all the indices $k_j\in I_j$
such that the side length of $Q_{j,k_j}$
is the first smallest among those of $\{Q_{1, k_1},\dots, Q_{\ell, k_\ell}\}$.
Therefore, the proof of \eqref{eq:e5} can be reduced to the following estimate:\,
\begin{eqnarray}\label{eq:e7}
&&\left|\left\{x\notin \bigcup_{i=1}^\ell \bigcup_{k\in I_i} Q_{i,k}^\ast:\, |T^{(j)}(b_1,\dots, b_\ell, g_{\ell+1},\dots, g_m)(x)|>\la\right\}\right|\\
&&\qquad\le C\,\la^{-\frac 1m}\,\left[\|T\|^p \alpha^{p-\frac1m}  + \alpha^{-\frac 1m}+ A \alpha^{1-\frac 1m}\right], \qquad 1\le j\le \ell. \notag
\end{eqnarray}

Without loss of generality, we may consider only the case $j=1$ in \eqref{eq:e7}.
Choose $t_{1,k_1}=(\sqrt{n}\ell(Q_{1,k_1}))^s$, where $s$ is the constant appearing in
\eqref{sizecondition1}. Write
\begin{eqnarray*}
&&T^{(1)}(b_1,\dots, b_\ell, g_{\ell+1},\dots, g_m)(x)\\
&&\quad= \sum_{k_2,\dots, k_\ell}\;\sum_{k_1\in \Theta_1} T(b_{1, k_1}\,,b_{2, k_2},\,\dots, b_{\ell, k_\ell},\, g_{\ell+1},\dots, g_m)(x)\\
&&\quad =
\sum_{k_2,\dots, k_\ell}\sum_{k_1\in \Theta_1} T(b_{1, k_1}-A_{t_{1,k_1}}^{(1)}b_{1, k_1}\,,b_{2, k_2},\,\dots, b_{\ell, k_\ell},\, g_{\ell+1},\dots, g_m)(x)\\
&&\qquad +\sum_{k_2,\dots, k_\ell}\sum_{k_1\in \Theta_1}T(A_{t_{1,k_1}}^{(1)}b_{1, k_1}\,,b_{2, k_2},\,\dots, b_{\ell, k_\ell},\, g_{\ell+1},\dots, g_m)(x)\\
&&\quad=: T^{(1,1)}(b_1,\dots, b_\ell, g_{\ell+1},\dots, g_m)(x)+T^{(1,2)}(b_1,\dots, b_\ell, g_{\ell+1},\dots, g_m)(x).
\end{eqnarray*}
Consider first the operator $ T^{(1,1)}$.  We get
\begin{eqnarray*}
&& |T^{(1,1)}(b_1,\dots, b_\ell, g_{\ell+1},\dots, g_m)(x)|\\
&&\quad \le \sum_{k_2,\dots, k_\ell}\;\sum_{k_1\in \Theta_1} |T(b_{1, k_1}-A_{t_{1,k_1}}^{(1)}b_{1, k_1}\,,b_{2, k_2},\,\dots, b_{\ell, k_\ell},\, g_{\ell+1},\dots, g_m)(x)|\\
&&\quad= \sum_{k_2,\dots, k_\ell}\;\sum_{k_1\in \Theta_1} \bigg(\int_{0}^{\infty}\bigg|
\int_{(\Rn)^m} [K_v(x, y_1,\dots, y_m)- K_{t_{1,k_1},v}^{(1)}(x, y_1,\dots, y_m)]\\
&&\quad\quad\quad\times\prod_{j=1}^\ell |b_{j, k_j}(y_j)|\;  \prod_{i=\ell+1}^m |g_{i}(y_i)| \,d\vec y\bigg|^2 \frac{dv}{v}\bigg)^{\frac{1}{2}}\\
&&\quad\le A \sum_{k_2,\dots, k_\ell}\;\sum_{k_1\in \Theta_1}\int_{(\Rn)^m} \left(\int_{0}^{\infty}
 \left|K_v(x, y_1,\dots, y_m)- K_{t_{1,k_1},v}^{(1)}(x, y_1,\dots, y_m)\right|^2 \frac{dv}{v}\right)^{\frac{1}{2}}\\
 &&\quad\quad\quad\times \;\prod_{j=1}^\ell |b_{j, k_j}(y_j)|\;  \prod_{i=\ell+1}^m |g_{i}(y_i)| \,d\vec y\\
&&\quad\le A \sum_{k_2,\dots, k_\ell}\;\sum_{k_1\in \Theta_1} \int_{(\Rn)^m} \frac{\left(\frac{(t_{1,k_1})^{1/s}}{\sum_{i=1}^m |x-y_i|}\right)^{\varepsilon}}{(\sum_{i=1}^m |x-y_i|)^{mn}}\;\prod_{j=1}^\ell |b_{j, k_j}(y_j)|\;  \prod_{i=\ell+1}^m |g_{i}(y_i)|\,d\vec y\\
&&\quad + A
\sum_{u=2}^m \sum_{k_2,\dots, k_\ell}\;\sum_{k_1\in \Theta_1}\int_{(\Rn)^m} \frac{\phi\left(\frac{|y_1-y_u|}{(t_{1,k_1})^{1/s}}\right)}{(\sum_{i=1}^m |x-y_i|)^{mn}}
\;\prod_{j=1}^\ell |b_{j, k_j}(y_j)|\;  \prod_{i=\ell+1}^m |g_{i}(y_i)|\,d\vec y\\
&&\quad=: Y_1(x)+Y_2(x),
\end{eqnarray*}
where we have used Assumption (H1) since $|x-y_1|\geq 2 t_{1,k_1}^{1/s}$.  For $Y_1(x)$, by the definitions of $t_{1,k_1}$ and $\Theta_1$, we see that
\begin{eqnarray*}
Y_1(x)
&&\le  A
\int_{(\Rn)^m} \;\prod_{j=1}^\ell \left(\sum_{k_j\in I_j}\;\frac{\left[\frac{\ell(Q_{j,k_j})}{ |x-y_j|}\right]^{\varepsilon/m}\;|b_{j, k_j}(y_j)|}{ |x-y_j|^{n}}\right)\; \\
&&\qquad \times
 \prod_{i=\ell+1}^m
  \left(\frac{\left[\frac{\ell(Q_{1,k_1})}{ \ell(Q_{1,k_1})+|x-y_i|}\right]^{\varepsilon/m}\;|g_{i}(y_i)|}{(\ell(Q_{1,k_1})+|x-y_i|)^{n}}\right)\,d\vec y.
\end{eqnarray*}
Therefore we obtain that for $\ell+1\le i\le m$,
$$
\int_{\Rn}
\frac{\left[\frac{\ell(Q_{1,k_1})}{ \ell(Q_{1,k_1})+|x-y_i|}\right]^{\varepsilon/m}\;|g_{i}(y_i)|}
{(\ell(Q_{1,k_1})+|x-y_i|)^{n}}\,dy_i
\le C M(g_i)(x)\le C (\alpha \lambda)^{1/m}.
$$
Recall that we have assumed that $x\notin \bigcup_{i=1}^m \bigcup_{k\in I_i} Q_{i,k}^\ast$.
Consequently, $x\notin Q_{j,k_j}^\ast$ for each $1\le j\le \ell$.
For such an $x$ and all $y\in Q_{j,k_j}$, it is easy to see that
$$ \frac 45|x-c_{j,k_j}|\le |x-y| \le \frac 43|x-c_{j,k_j}|,$$
where $c_{j,k_j}$ denotes the center of the cube $Q_{j,k_j}$.
Furthermore, we have
\begin{eqnarray*}
&&\int_{\Rn} \sum_{k_j\in I_j}\;\frac{\left[\frac{\ell(Q_{j,k_j})}{ |x-y_j|}\right]^{\varepsilon/m}\;|b_{j, k_j}(y_j)|}{ |x-y_j|^{n}}\,dy_j\\
&&\quad\le C \sum_{k_j\in I_j}\;\frac{\left[\frac{\ell(Q_{j,k_j})}{ \frac{4}{5}|x-c_{j,k_j}|}\right]^{\varepsilon/m}\;\|b_{j, k_j}\|_{L^1(\Rn)}}{ |x-c_{j,k_j}|^{n}}\\
&&\quad\le C (\alpha\la)^{1/m}\sum_{k_j\in I_j}\;\left[\frac{\ell(Q_{j,k_j})}{ \frac{4}{5}|x-c_{j,k_j}|}\right]^{\varepsilon/m}
\frac{\;|Q_{j, k_j}|}{ |x-c_{j,k_j}|^{n}}.
\end{eqnarray*}
Define the function
\begin{equation}\label{Marc}
M_{j, 1/m}(x):= \sum_{k_j\in I_j}\;\left[\frac{\ell(Q_{j,k_j})}{ \frac{4}{5}|x-c_{j,k_j}|}\right]^{\varepsilon/m}
\frac{\;|Q_{j, k_j}|}{ |x-c_{j,k_j}|^{n}}.
\end{equation}
Then, we have proved that
\begin{equation}\label{eq:Y1}
Y_1(x)\le C A \alpha\lambda \prod_{j=1}^\ell M_{j, 1/m}(x).
\end{equation}
Next, we shall use the following estimate in \cite[p. 240]{MN}:\,
\begin{equation}
\int_{x\notin  \bigcup_{i=1}^m \bigcup_{k\in I_i} Q_{i,k}^\ast}
M_{j, 1/m}(x)\,dx \le C \sum_{k_j\in I_j} |Q_{j,k_j}|.
\end{equation}
By this and H\"older's inequality, we get
\begin{eqnarray}\label{eq:e9}
&&\left|\left\{x\notin \bigcup_{i=1}^\ell \bigcup_{k\in I_i} Q_{i,k}^\ast:\, |Y_1(x)|>\la/4\right\}\right|\\
&&\quad \le \left|\left\{x\notin \bigcup_{i=1}^\ell \bigcup_{k\in I_i} Q_{i,k}^\ast:\, A \alpha\prod_{j=1}^\ell M_{j, 1/m}(x)>C\right\}\right|\notag\\
&&\quad\le C
\int_{x\notin  \bigcup_{i=1}^m \bigcup_{k\in I_i} Q_{i,k}^\ast}
\left( A \alpha\prod_{j=1}^\ell M_{j, 1/m}(x)\right)^{1/\ell}\,dx\notag\\
&&\quad \le C
\left(A \alpha \prod_{j=1}^\ell \int_{x\notin  \bigcup_{i=1}^m \bigcup_{k\in I_i} Q_{i,k}^\ast}
 M_{j, 1/m}(x)\,dx\right)^{1/\ell}\notag\\
 &&\quad \le C (A \alpha )^{1/\ell} \left(\prod_{j=1}^\ell
 \sum_{k_j\in I_j} |Q_{j,k_j}| \right)^{1/\ell} \notag\\
 &&\quad \le C (\alpha \la)^{-1/m}, \notag
\end{eqnarray}
where the last step is by (iii) and the fact $A\alpha<1$.

Now we consider  $Y_2(x)$. Indeed, the estimate of $Y_2$ is exactly the term $T_2^{(11)}$ in \cite[p.\,2100]{DGY}, which provides us the following estimate:
\begin{equation*}\begin{split}
  Y_2(x)
\le C A (\alpha\lambda)^{\ell/m} &\sum_{k=\ell+1}^m
\left(\prod_{j=2}^\ell \mathcal J_{j, \frac{n}{m-1}} (x) \right)\;
\left(\prod_{i\neq k,\; \ell+1\le i\le m}
M(g_i)(x)\right)\; \\ &\times\left(
M(g_k)(x)+(\alpha\lambda)^{1/m} \mathcal J_{1, \frac{n}{m-1}} (x)
\right),
\end{split}\end{equation*}
where for any $\epsilon>0$ and $1\le j\le m$ the function $\mathcal J_{j,\epsilon}(x)$
is given by
\begin{eqnarray*}
\mathcal J_{j,\epsilon}(x)
= \sum_{k_j\in I_j} \frac{[\ell(Q_{j,k_j})]^{n+\epsilon}}
{[\ell(Q_{j,k_j})+|x-c_{j, k_j}|]^{n+\epsilon}}
\end{eqnarray*}
and $M$  is usual the Hardy-Littlewood maximal operator.
\par
It is known from \cite[(2.6)]{DGY} that for any $p\in(n/(n+\epsilon),\infty)$,
\begin{equation}\label{eq:e222}
\|\mathcal J_{j,\epsilon}\|_{L^p(\Rn)} \le C \left(
\sum_{k_j\in I_j}|Q_{j,k_j}|\right)^{1/p}.
\end{equation}
By (iii), we further have that
\begin{equation}\label{eq:e12}
\|\mathcal J_{j,\epsilon}\|_{L^2(\Rn)} \le C \left(
\sum_{k_j\in I_j}|Q_{j,k_j}|\right)^{1/2}
\le C (\alpha\la)^{-1/(2m)}.
\end{equation}
By the $L^2$-boundedness of $M$, property (i) and \eqref{eq:xx1}, we have that for all $1\le i\le m$,
\begin{equation}\label{eq:e13}
\|M(g_i)\|_{L^2(\Rn)}
\le C\|g_i\|_{L^2(\Rn)} \le C (\alpha\la)^{1/(2m)} \|g_i\|_{L^1(\Rn)}^{1/2}
\le C (\alpha\la)^{1/(2m)}.
\end{equation}
From \eqref{eq:e12} and \eqref{eq:e13}, it follows that
\begin{eqnarray*}
&&\left|\left\{x\notin \bigcup_{i=1}^\ell \bigcup_{k\in I_i} Q_{i,k}^\ast:\, |Y_2(x)|>\la/4\right\}\right|\\
&&\quad \le \int_{\Rn}  \left(\frac{Y_2(x)}{\la/4}\right)^{\frac 2{m-1}}\,dx \notag\\
&&\quad \le
C
\left(\frac{A (\alpha\lambda)^{\ell/m} }{\la}\right)^{\frac 2{m-1}}
\sum_{k=\ell+1}^m
\int_{\Rn}
\left(\prod_{j=2}^\ell \mathcal J_{j, \frac{n}{m-1}} (x) \right)^{\frac 2{m-1}}\;\notag\\
&&\qquad\times
\left(\prod_{i\neq k,\; \ell+1\le i\le m}
M(g_i)(x)\right)^{\frac 2{m-1}}\; \left(
M(g_k)(x)+(\alpha\lambda)^{1/m} \mathcal J_{1, \frac{n}{m-1}} (x)
\right)^{\frac 2{m-1}}\,dx\notag.
\end{eqnarray*}
In the last formula, there are $m-1$ factors inside the integrand.
Since $\ell\ge1$, so for the special case $m=2$,   we have only the last factor
$(
M(g_k)(x)+(\alpha\lambda)^{1/m} \mathcal J_{1, \frac{n}{m-1}} (x)
)^2$ inside the integrand.
We shall use the inequality:
\begin{equation}\label{eq:e14}
\int_{\Rn} \left[F_1(x) \cdots F_{m-1}(x)\right]^{\frac 2{m-1}}\,dx
\le \prod_{j=1}^{m-1}\left(\int_{\Rn} [F_j(x)]^2\,dx\right)^{\frac1{m-1}}.
\end{equation}
This inequality holds with ``$=$" when  $m=2$,
and it follows from
 H\"older's inequality when $m\ge3$.
 Therefore, by \eqref{eq:e14}, we have
 \begin{eqnarray}\label{eq:e15}
&&\left|\left\{x\notin \bigcup_{i=1}^\ell \bigcup_{k\in I_i} Q_{i,k}^\ast:\, |Y_2(x)|>\la/4\right\}\right|\\
&& \quad\le
C
\left(\frac{A (\alpha\lambda)^{\ell/m} }{\la}\right)^{\frac 2{m-1}}
\sum_{k=\ell+1}^m
\left(\prod_{j=2}^\ell \| \mathcal J_{j, \frac{n}{m-1}}\|_{L^2(\Rn)}^{\frac{2}{m-1}} \right) \notag\\
&&\qquad\times\left(\prod_{i\neq k,\; \ell+1\le i\le m} \| M(g_i)\|_{L^2(\Rn)}^{\frac{2}{m-1}}  \right)\notag\\
&&\qquad\times
\left(
\|M(g_k)\|_{L^2(\Rn)} +(\alpha\lambda)^{1/m}
\| \mathcal J_{1, \frac{n}{m-1}}\|_{L^2(\Rn)}
\right)^{\frac{2}{m-1}}  \notag\\
&&\quad \le C
\left(\frac{A (\alpha\lambda)^{\ell/m} }{\la}\right)^{\frac 2{m-1}}
\left( (\alpha\lambda)^{-\frac{\ell-1}{2m}}\;\cdot (\alpha\lambda)^{\frac{m-\ell-1}{2m}} \;\cdot
(\alpha\lambda)^{\frac{1}{2m}} \right)^{\frac 2{m-1}}
\notag\\
&&\quad
= C  (A\alpha)^{\frac 2{m-1}} (\alpha\la)^{-1/m}, \notag
\end{eqnarray}
which is bounded by $CA\alpha (\alpha\la)^{-1/m}$ since $A\alpha<1$ and $\frac 2{m-1}\ge1$.
\par The relations \eqref{eq:e9} and \eqref{eq:e15} together yield that
\begin{eqnarray*}
&&\left|\left\{x\notin \bigcup_{i=1}^\ell \bigcup_{k\in I_i} Q_{i,k}^\ast:\, |T^{(1,1)}(b_1,\dots, b_\ell, g_{\ell+1},\dots, g_m)(x)|>\la/2\right\}\right| \le C\,(\la \alpha)^{-\frac 1m}\,\left[1+ A \alpha\right]. \notag
\end{eqnarray*}
Thus, to obtain \eqref{eq:e7}, it remains to prove
\begin{eqnarray}\label{eq:e16}
\;\left|\left\{x\notin \bigcup_{i=1}^\ell \bigcup_{k\in I_i} Q_{i,k}^\ast:\, |T^{(1,2)}(b_1,\dots, b_\ell, g_{\ell+1},\dots, g_m)(x)|>\la/2\right\}\right|\le C(\alpha \lambda)^{-1/m}.
\end{eqnarray}
To handle the rest of the proof we use \eqref{eq:e222}, Chebychev's inequality and the $L^{q_1}\times \cdots \times L^{q_m}\rightarrow L^q$ boundedness of $T$ to get the desired result.   This concludes the proof of the Theorem.
\end{proof}

\section{Weighted  estimates for $T$}
%

In this section, we will study the multiple-weighted normal inequalities and weak-type estimates.
%
Our main results in this section can be stated as follows.
\begin{theorem}\label{thm:22}
 Let $T$ be a multilinear operator in $m-GSFO(A,s,\eta,\varepsilon)$ with a kernel  satisfies Assumption (H2). Let $\frac{1}{p}=\frac{1}{p_1}+\ldots+\frac{1}{p_m}$ for $1\leq p_1,\ldots,p_{m}< \infty$ and $\omega\in A_p$ with $p\geq 1.$ , then the following hold:

 \begin{enumerate}
   \item If there is no $p_i=1$, then $\|T(\vec{f}) \|_{L^{p}(\omega)}\leq C \prod_{i=1}^{m} \|f_i\|_{L^{p_i}(\omega)} $\\
   \item If there is a $p_i=1$, then $\|T(\vec{f}) \|_{L^{p,\infty}(\omega)}\leq C \prod_{i=1}^{m} \|f_i\|_{L^{p_i}(\omega)} $
 \end{enumerate}
\end{theorem}
The proof of Theorem \ref{thm:22} is similar to \cite[Cor. 4.2]{LOPTT}, which is based on the following lemmas.

\begin{lemma}(\cite{fefferman})\label{lw1}
Let $0<p,\delta<\infty$ and $\omega$ be any Mackenhoupt  $A_\infty$ weight. Then there exists a  constant C independent of $f$ such that the inequality
\begin{equation}\label{21}
\int_{\mathbb{R}^n}(M_\delta f(x))^p\omega(x)dx\leq C
\int_{\mathbb{R}^n}(M^\sharp_\delta f(x))^p\omega(x)dx
\end{equation}
holds for any function such that the left-hand side is finite.

Moreover, if $\varphi: (0,\infty)\rightarrow (0,\infty)$ be a doubling, that is,
 $\varphi(2a)\leq C \varphi(a)$ for $a>0$. Then, there exists a constant C depending upon the $A_\infty$ condition of $\omega$ and doubling condition of $\varphi$ such that
\begin{equation}\label{20}
\sup_{\lambda>0}\varphi(\lambda)\omega(\{y\in \mathbb{R}^n: M_\delta
f(y)>\lambda\})\leq C \sup_{\lambda>0}\varphi(\lambda)\omega(\{y\in
\mathbb{R}^n: M^\sharp_\delta f(y)>\lambda\})
\end{equation}
for any function such that the left-hand side is finite.

\end{lemma}
Similar as \cite[Lemma 4.2]{XY}, we can easily get
\begin{lemma}\label{lw2}
     Let $T$ be a multilinear operator in $m-GSFO(A,s,\eta,\varepsilon)$. Suppose supp $f_i \subset B(0,R)$,  then there is a constant $C<\infty$  such that for $|x|>2R,$

$$  T(\vec{f})(x)\leq C  \prod_{j=1}^m M f_j(x).$$
\end{lemma}

\begin{lemma}\label{lw3} Let $ 0<\delta<1/m$ and let $T$ be a multilinear operator in $m-GSFO(A,s,\eta,\varepsilon)$ with a kernel  satisfies Assumption (H2). Then there exists a constant $C$ such that for any bounded and compact supported function $f_i,i=1,\dots,m.$
$$
M_{\delta}^{\#}T\vec{f}(x)\leq C \prod_{j=1}^m M f_j(x).
$$
\end{lemma}

\begin{proof} Fix a point $x\in \mathbb{R}^n$ and a cube $Q$ containing $x$. For $0<\delta<1/ m,$ we need to show there exists a constant $c_Q$ such that
$$\left( \frac{1}{|Q|}\int_{Q}\left| T(\vec{f})(z)-c_Q\right|^\delta dz\right)^{1/\delta}\leq C \prod_{j=1}^m M f_j(x).$$

Let $c_{Q,t}=\sum_{\vec{\alpha},\vec{\alpha}\neq \vec{0}} \int_{\mathbb{R}^{nm}}K_t(x,y_1,\cdots,y_m)\prod_{j=1}^m f_j(y_j)d\vec{y}$, where $\vec{\alpha}=(\alpha_1,\cdots,\alpha_m)$ with $\alpha_i=0$ or $\infty$ and let $c_Q=\left( \int_0^\infty |c_{Q,t}|^2\frac{dt}{t} \right)^{1/2}.$
\begin{eqnarray*}
 &&\left( \frac{1}{|Q|}\int_{Q}\left| T(\vec{f})(z)-c_Q\right|^\delta dz\right)^{1/\delta}\\
&\leq  &C \left( \frac{1}{|Q|}\int_{Q}\left(\int_0^\infty\left| \int_{\mathbb{R}^{nm}}K_t(z,\vec{y})\prod_{j=1}^m f_j^0(y_j)d\vec{y}\right|^2\frac{dt}{t}\right)^{\delta/2} dz\right)^{1/\delta}\\
&\quad +  &C \sum_{\vec{\alpha},\vec{\alpha}\neq \vec{0}}\left( \frac{1}{|Q|}\int_{Q}\left(\int_0^\infty\left| \int_{\mathbb{R}^{nm}}\left(K_t(z,\vec{y})-K_t(x,\vec{y})\right)\prod_{j=1}^m |f_j^{\alpha_j}(y_j)|d\vec{y}\right|^2\frac{dt}{t}\right)^{\delta/2} dz\right)^{1/\delta}\\
&:=& I_{\vec{0}} + C\sum_{\vec{\alpha}\neq \vec{0}} I_{\vec{\alpha}}.
\end{eqnarray*}

By using Kolmogorov' inequality and Theorem \ref{thm:endpoint}, we get $I_{\vec{0}}\leq C \prod_{j=1}^m M f_j(x).$
 \par
 Now, we turn our attention to the integral
  $$\int_{\mathbb{R}^{nm}}\left(\int_0^\infty\left| K_v(z,\vec{y})-K_v(x,\vec{y})\right|^2\frac{dv}{v}\right)^{1/ 2}\prod_{j=1}^m| f_j^{\alpha_j}(y_j)|d\vec{y},$$

which can be decomposed as
 \begin{eqnarray*}
  &&\int_{\mathbb{R}^{nm}}\left(\int_0^\infty\left| K_v(z,\vec{y})-K_v(x,\vec{y})\right|^2\frac{dv}{v}\right)^{1/ 2}\prod_{j=1}^m| f_j^{\alpha_j}(y_j)|d\vec{y}\\
   &\leq & \int_{\mathbb{R}^{nm}}\left(\int_0^\infty\left| K_v(z,\vec{y})-K_{t,v}^{(0)}(z,\vec{y})\right|^2\frac{dv}{v}\right)^{1/ 2}\prod_{j=1}^m| f_j^{\alpha_j}(y_j)|d\vec{y}\\
   &\quad + & \int_{\mathbb{R}^{nm}}\left(\int_0^\infty\left| K_{t,v}^{(0)}(z,\vec{y})-K_v(x,\vec{y})\right|^2\frac{dv}{v}\right)^{1/ 2}\prod_{j=1}^m| f_j^{\alpha_j}(y_j)|d\vec{y}\\
   &=:& I_{1}+I_{2}.
\end{eqnarray*}
By \eqref{smoothnessforkernel1} in Assumption (H2) we obtain that
\begin{eqnarray*}
   I_{1}&\leq & C \prod_{j\in \{j_1,\dots,j_l\}}\int_{Q^*}|f_j(y_j)|dy_j\bigg( \int_{(\mathbb{R}^n\setminus Q^*)^{m-l}}\frac{t^{\varepsilon/ s}\prod_{j\notin \{j_1,\dots,j_l\}}|f_j(y_j)|dy_j}{(\sum_{j\notin \{j_1,\dots,j_l\}}|z-y_j|)^{mn+\varepsilon}}\\
   &\quad+& \int_{(\mathbb{R}^n\setminus Q^*)^{m-l}}\frac{ \prod_{j\notin \{j_1,\dots,j_l\}}|f_j(y_j)|dy_j}{(\sum_{j\notin \{j_1,\dots,j_l\}}|z-y_j|)^{mn }}\bigg)\\
   &\leq & C \prod_{j\in \{j_1,\dots,j_l\}}\int_{Q^*}|f_j(y_j)|dy_j \bigg( \sum_{k=1}^\infty \frac{|Q^*|^{\varepsilon / n}}{(2^k|Q^*|^{1/ n})^{mn+\varepsilon}} \int_{(2^{k+1} Q^*)^{m-l}}\prod_{j\notin \{j_1,\dots,j_l\}}|f_j(y_j)|dy_j\\
     &\quad+& \sum_{k=1}^\infty \frac{1}{(2^k|Q^*|^{1/ n})^{mn}} \int_{(2^{k+1} Q^*\setminus 2^{k} Q^*)^{m-l}}\prod_{j\notin \{j_1,\dots,j_l\}}|f_j(y_j)|dy_j\bigg)\\
     &\leq & C \prod_{j=1}^m M f_j(x).
\end{eqnarray*}

$I_{2}$ can be estimated in a similar way, which is also dominated by $C\prod_{j=1}^m M f_j(x).$
Then, by Minkowski's inequality and above estimates, we get
\begin{eqnarray*}
  I_{\vec{\alpha}}&\leq & C\left( \frac{1}{|Q|}\int_{Q}\left( \int_{\mathbb{R}^{nm}}\left(\int_0^\infty\left| K_t(z,\vec{y})-K_t(x,\vec{y})\right|^2\frac{dt}{t}\right)^{1/ 2}\prod_{j=1}^m| f_j^{\alpha_j}(y_j)|d\vec{y}\right)^{\delta} dz\right)^{1/\delta}\\
  &\leq & C \prod_{j=1}^m M f_j(x).
\end{eqnarray*}
This concludes the proof.
\end{proof}

\section{Weighted estimates for $T^*$}
In this section, we will study the multilinear maximal square function $T^*$

\begin{equation*}
\begin{split}
T^*(\vec{f})(x)=\left( \int_{0}^\infty \sup_{\delta>0}\left|\int_{\sum_{i=1}^m|x-y_i|^2>\delta^2}K_{v}(x,y_1,\dots,y_m) \prod_{j=1}^mf_{j}(y_j)d\vec{y}\right|^2\frac{dv}{v}\right)^{\frac 12}.
\end{split}
\end{equation*}

It should be pointed out that there is another kind of multilinear maximal square function $T^{**}$ given by

\begin{equation*}
\begin{split}
T^{**}(\vec{f})(x)=\sup_{\delta>0}\left( \int_{0}^\infty \left|\int_{\sum_{i=1}^m|x-y_i|^2>\delta^2}K_{v}(x,y_1,\dots,y_m) \prod_{j=1}^mf_{j}(y_j)d\vec{y}\right|^2\frac{dv}{v}\right)^{\frac 12}.
\end{split}
\end{equation*}

It is obvious that $T^{**}(\vec{f})(x)\leq T^*(\vec{f})(x)$. Thus, it is more meaningful to give some estimates for operator $T^*.$
In this following, we establish some multiple-weighted normal inequalities and weak-type estimates for $T^*$.


\begin{theorem}\label{lemma2}  Let $\frac{1}{p}=\frac{1}{p_1}+\cdots +\frac{1}{p_m}$ and $\omega\in A_p$ with $p\geq 1.$ Let $T$ be a multilinear operator in $m-GSFO(A,s,\eta,\varepsilon)$ with a kernel satisfies Assumptions (H2) and (H3). Then the following inequality holds:
\par
(i) If $1<p_1, \cdots, p_m<\infty,$ then
$$||T^*\vec{f}||_{L^p(\omega)}\leq C\prod
\limits_{i=1}^{m}||f_i||_{L^{p_i}(\omega)}.$$
\par
(ii) If  $1\leq p_1, \cdots, p_m<\infty, $ then
$$||T^*\vec{f}||_{L^{p,\infty}(\omega)}\leq C\prod
\limits_{i=1}^{m}||f_i||_{L^{p_i}(\omega)}.$$
\end{theorem}
To prove Theorem \ref{lemma2} we need some lemmas.

\begin{lemma}\label{31}Let $T$ be a multilinear operator in $m-GSFO(A,s,\eta,\varepsilon)$ with a kernel satisfies Assumptions (H2) and (H3).  For any $\eta>0$, there is a constant $C<\infty$ depending on $\eta$ such that for all $\vec{f}$ in any product of $L^{q_j}(\mathbb{R}^n)$ spaces, with $1\leq q_j<\infty,$ then the following inequality hold for all $x\in \mathbb{\mathbb{R}}^n$

$$T^*(\vec{f})(x)\leq C \left( M_\eta(T(\vec{f}))(x)+\prod_{j=1}^m M(f_j)(x)\right).$$
\end{lemma}
\begin{proof} For a fixed point $x$ and  $\delta>0$ we denote by
$U_\delta(x)=\{ \vec{y} : \sum_{i=1}^m|x-y_i|^2 < \delta^2\}$ and $V_\delta(x)=\{ \vec{y}  : \inf_{1\leq j\leq m}|y_j-x| \geq \delta\}$. It is clear that
\begin{eqnarray*}
\left| T^*(\vec{f})(x)\right|&\leq & \left( \int_{0}^\infty \sup_{\delta>0}\left|\int_{\left(U_\delta(x) \cup V_\delta(x)\right)^c }K_v(x,\vec{y})\prod_{i=1}^mf_i(y_i)dy_i\right|^2\frac{dv}{v}\right)^{\frac 12}\\
&+&\left( \int_{0}^\infty\sup_{\delta>0}\left|\int_{V_\delta(x)}K_v(x,\vec{y})\prod_{i=1}^mf_i(y_i)dy_i\right|^2\frac{dv}{v}\right)^{\frac 12}.
\end{eqnarray*}
By using the size condition and Minkowski's inequality, we get
\begin{eqnarray}
 \begin{split}\label{t1}
&\left( \int_{0}^\infty \left|\int_{\left(U_\delta(x) \cup V_\delta(x)\right)^c}K_v(x,\vec{y})\prod_{i=1}^mf_i(y_i)dy_i\right|^2\frac{dv}{v}\right)^{\frac 12}\\
&\leq  C \left|\int_{\left(U_\delta(x) \cup V_\delta(x)\right)^c}\left( \int^{\infty}_0 \left|K_v(z,\vec{y}) \right|^2\frac{dv}{v}\right)^2\prod_{i=1}^m|f_i(y_i)|dy_i\right|\\
&\leq  C \left|\int_{\left(U_\delta(x) \cup V_\delta(x)\right)^c}\frac{A}{(\delta+\sum_{i=1}^m |y_i-x|)^{mn}}\prod_{i=1}^m|f_i(y_i)|dy_i\right|\\
&\leq  C\prod_{j\in \{j_1,\dots, j_l\}} \frac{1}{\delta^n}\int_{|x-y_j|<\delta}|f_i(y_i)|dy_i  \\
&\quad \times \prod_{j\notin \{j_1,\dots, j_l\}} \int_{|x-y_j|\geq\delta}\frac{\delta^{(ln)/(m-l)}}{|x-y_j|^{(mn)/(m-l)}}|f_i(y_i)|dy_i\\
&\leq  C \prod_{j=1}^m M(f_j)(x).
 \end{split}
 \end{eqnarray}
We are ready to estimate the second term.  Fix $\delta>0$ and let $B(x,\delta / 2)$ be the ball of center $x$ and radius $\delta / 2$.
Set $\vec{f}_0=\left( f_1\chi_{B(z,\delta)},\cdots,  f_m\chi_{B(z,\delta)}\right),$ for any $z\in B(x,\frac \delta 2)$, we have
\begin{eqnarray*}
&&\tilde{T}_\delta(\vec{f})(z)\\
&:=&\left( \int_{0}^\infty \left|\int_{V_\delta(z)}K_v(z,\vec{y})\prod_{i=1}^mf_i(y_i)dy_i\right|^2\frac{dv}{v}\right)^{\frac 12}\\
&=&\left( \int_{0}^\infty \left|\int_{(\mathbb{R}^{n})^m}K_v(z,\vec{y})\left(\prod_{i=1}^mf_i(y_i)-\prod_{i=1}^mf_i(y_i)\chi_{Q(z,\delta)} -\prod_{i=1}^mf_i(y_i)\chi_{\left(U_\delta(z) \cup V_\delta(z)\right)^c}\right) dy_i\right|^2\frac{dv}{v}\right)^{\frac 12}\\
&\leq &T(\vec{f})(z)+T(\vec{f}_0)(z)+\left( \int_{0}^\infty \left|\int_{\left(U_\delta(z) \cup V_\delta(z)\right)^c}K_v(x,\vec{y})\prod_{i=1}^mf_i(y_i)dy_i\right|^2\frac{dv}{v}\right)^{\frac 12}.\\
\end{eqnarray*}
The condition $z\in B(x,\delta / 2)$ and $|y_j-z|>\delta$ imply  $|y_j-z|/ 2\leq |y_j-x|\leq 2|y_j-z|.$ By using an argument as that in \eqref{t1}, we deduce that
$$\left( \int_{0}^\infty \left|\int_{\left(U_\delta(z) \cup V_\delta(z)\right)^c}K_v(x,\vec{y})\prod_{i=1}^mf_i(y_i)dy_i\right|^2\frac{dv}{v}\right)^{\frac 12}\leq C\prod_{j=1}^m M(f_j)(x).$$
In order to estimate  $\tilde{T}_\delta(\vec{f})(x)-\tilde{T}_\delta(\vec{f})(z)$,  we introduce an operator $\mathcal{T}$ which is given by
$$\mathcal{T}_\delta f(z)=\left( \int_{0}^\infty \left|\int_{V_\delta(z)}\int_{(\mathbb{R}^n)}K_v(y,\vec{y})a_t(z,y)dy\prod_{i=1}^mf_i(y_i)dy_i\right|^2\frac{dv}{v}\right)^{\frac 12}.$$
 where $t=(\delta / 4)^{s}$ and $s$ is a constant in $\eqref{smallcondition2}$.
We decompose it into
\begin{eqnarray*}
&&\left|\tilde{T}_\delta(\vec{f})(x)-\tilde{T}_\delta(\vec{f})(z)\right|\\
&\leq & \left|\mathcal{T}_\delta(\vec{f})(z)-\tilde{T}_\delta(\vec{f})(z)\right|+ \left|\mathcal{T}_\delta(x)-\mathcal{T}_\delta(\vec{f})(z)\right|+ \left|\tilde{T}_\delta(\vec{f})(x)-\mathcal{T}_\delta(\vec{f})(x)\right|\\
&= & I+ II+III.
 \end{eqnarray*}
 For the first term we use Assumption (H2).  First note that $|z-y_j|\geq \delta=4t^{1/ s}.$
  This, in combination with the fact that $\textsf{supp}\, \phi \subset [-1,1],$ yields that $\phi(\frac{|z-y_j|}{t^{1/ s}})=0.$
 By the facts that $z\in B(x,\delta / 2)$ and $|y_j-z|\geq \delta$, we obtain
 $$ \frac{|y_j-z|}{2}\leq |y_j-x|\leq 2|y_j-z|.$$
  By \eqref{smoothnessforkernel1} in Assumption (H2) we obtain that

 \begin{eqnarray*}
&&\left|\mathcal{T}_\delta(\vec{f})(z)-\tilde{T}_\delta(\vec{f})(z)\right|\\
&=&\bigg|\left( \int_{0}^\infty \left|\int_{V_\delta(z)}\int_{(\mathbb{R}^n)}K_v(y,\vec{y})a_t(z,y)dy\prod_{i=1}^mf_i(y_i)dy_i\right|^2\frac{dv}{v}\right)^{\frac 12}\\
&&-\left( \int_{0}^\infty \left|\int_{V_\delta(z)} K_v(z,\vec{y})\prod_{i=1}^mf_i(y_i)dy_i\right|^2\frac{dv}{v}\right)^{\frac 12}\bigg|\\
&\leq&\left( \int_{0}^\infty \left|\int_{V_\delta(z)}\left(K_{t,v}^{(0)}(z,\vec{y})-  K_v(z,\vec{y})\right)\prod_{i=1}^mf_i(y_i)dy_i\right|^2\frac{dv}{v}\right)^{\frac 12}\\
&\leq & \left( \int_{0}^\infty \left|\int_{V_\delta(z)}\left| K_v(z,\vec{y})-K_{t,v}^{(0)}(z,\vec{y})\right|\prod_{i=1}^m|f_i(y_i)|dy_i\right|^2\frac{dv}{v}\right)^{\frac 12}  \\
&\leq & \int_{V_\delta(z)} \left(\int_{0}^\infty \left| K_v(z,\vec{y})-K_{t,v}^{(0)}(z,\vec{y})\right|^2\frac{dv}{v}\right)^{\frac 12}\prod_{i=1}^m|f_i(y_i)|dy_i \\
&\leq & \int_{V_\delta(z)} \frac{\delta^\epsilon}{(|z-y_1|+\dots +|z-y_m|)^{mn+\epsilon}}\prod_{i=1}^m|f_i(y_i)|dy_i\\
&\leq & \int_{|y_1-x|\geq \delta  ,\dots, |y_m-x|\geq \delta  } \frac{\delta^\epsilon}{(|z-y_1|+\dots +|z-y_m|)^{mn+\epsilon}}\prod_{i=1}^m|f_i(y_i)|dy_i\\
&\leq & \prod_{j=1}^m M(f_j)(x).
 \end{eqnarray*}
 Similarly we get  $III \leq  \prod_{j=1}^m M(f_j)(x).$
 \par We now turn to the second term $II.$
  \begin{eqnarray*}
 &&\left|\mathcal{T}_\delta(\vec{f})(x)-\mathcal{T}_\delta(\vec{f})(z)\right|\\
 &\leq & C \bigg|  \left( \int_{0}^\infty \left|\int_{ V_\delta(x)\setminus V_\delta(z)} \int_{\mathbb{R}^n} a_t(x,y)K_v(y,\vec{y})dy\prod_{i=1}^mf_i(y_i)dy_i\right|^2\frac{dv}{v}\right)^{\frac 12}  \bigg|\\
  &+ &C\bigg|\left( \int_{0}^\infty \left|\int_{ V_\delta(z)}\int_{\mathbb{R}^n} a_t(x,y) K_v(y,\vec{y})dy\prod_{i=1}^mf_i(y_i)dy_i\right|^2\frac{dv}{v}\right)^{\frac 12} \\
   && - \left( \int_{0}^\infty \left|\int_{ V_\delta(z)}\int_{\mathbb{R}^n} a_t(z,y) K_v(y,\vec{y})dy\prod_{i=1}^mf_i(y_i)dy_i\right|^2\frac{dv}{v}\right)^{\frac 12} \bigg|\\
   &=&II_1+II_2.
 \end{eqnarray*}
As in \cite{duonggong}, since $z\in B(x, \delta / 2),$ we use the fact $V_{3\delta / 2}(x)\subset V_{\delta }(z)$ and $V_{\delta }(x)\setminus V_{\delta }(z)\subset V_{\delta }(x)\setminus V_{3\delta / 2}(x) \subset (\mathbb{R}^n)^m \setminus (U_\delta(x)\cup V_{3\delta/ 2}(x)).$ By \eqref{sizeforkernel2} in Assumption (H2), we obtain
  \begin{eqnarray*}
 II_1&\leq &C \int_{(\mathbb{R}^n)^m\setminus (U_\delta(x)\cup V_{3\delta/ 2}(x))}  \left( \int_{0}^\infty \left|\int_{ \mathbb{R}^n}K_v(y,\vec{y})a_t(z,y)dy\right|^2\frac{dv}{v}\right)^{\frac 12} \prod_{i=1}^m |f_i(y_i)|dy_i  \\
 &=& \int_{(\mathbb{R}^n)^m\setminus (U_\delta(x)\cup V_{3\delta/ 2}(x))}  \left( \int_{0}^\infty \left|K_{t,v}^{(0)}(z,\vec{y})\right|^2\frac{dv}{v}\right)^{\frac 12} \prod_{i=1}^m |f_i(y_i)|dy_i\\
 &\leq &C \int_{(\mathbb{R}^n)^m\setminus (U_\delta(x)\cup V_{3\delta/ 2}(x))}  \frac{A}{(\delta+\sum_{i=1}^m |y_i-x|)^{mn}}\prod_{i=1}^mf_i(y_i)dy_i\\
 &\leq &C\prod_{j=1}^m M(f_j)(x).
 \end{eqnarray*}

For $II_2$, we use \eqref{smoothnessforkernel2} in Assumption (H3), and a similar argument as in term $I$ to deduce that  for $z\in B(x,\delta / 2)$ we have

   \begin{eqnarray*}
 II_2&=&\bigg|\left( \int_{0}^\infty \left|\int_{V_\delta(z)}K_{t,v}^{(0)}(x,\vec{y})\prod_{i=1}^mf_i(y_i)dy_i\right|^2\frac{dv}{v}\right)^{\frac 12}\\
 &&-\left( \int_{0}^\infty \left|\int_{V_\delta(z)}K_{t,v}^{(0)}(z,\vec{y})\prod_{i=1}^mf_i(y_i)dy_i\right|^2\frac{dv}{v}\right)^{\frac 12}\bigg|\\
 &\leq &C  \left( \int_{0}^\infty \left|\int_{V_\delta(z)}(K_{t,v}^{(0)}(x,\vec{y})-K_{t,v}^{(0)}(z,\vec{y}))\prod_{i=1}^mf_i(y_i)dy_i\right|^2\frac{dv}{v}\right)^{\frac 12}\\
 &\leq &C  \bigg|\int_{V_\delta(z)} \left( \int_{0}^\infty \left| K_{t,v}^{(0)}(x,\vec{y})-K_{t,v}^{(0)}(z,\vec{y})\right|^2\frac{dv}{v}\right)^{\frac 12} \prod_{i=1}^mf_i(y_i)dy_i\bigg|\\
  &\leq &C \int_{|y_1-x|\geq \delta ,\dots, |y_m-x|\geq \delta  } \frac{\delta^\epsilon}{(|z-y_1|+\dots +|z-y_m|)^{mn+\epsilon}}\prod_{i=1}^m|f_i(y_i)|dy_i\\
  &\leq & C\prod_{j=1}^m M(f_j)(x).
 \end{eqnarray*}
 Then we obtain $$II\leq C \prod_{j=1}^m M(f_j)(x).$$
 Thus, for any $z\in B(x,\delta / 2),$ we have
 $$\tilde{T}_\delta(\vec{f})(x)\leq C \prod_{j=1}^m M(f_j)(x) + |T(\vec{f})(z)|+|T(\vec{f}_0)(z)|.$$

Fix $0<\eta<1/ m.$ Raising the above inequality to the power $\eta$, integrating over $z\in B(x,\delta / 2),$ and dividing by $B$ we obtain
   \begin{eqnarray*}
\left|\tilde{T}_\delta(\vec{f})(x)\right|^{\eta}&\leq &C \left(\prod_{j=1}^m M(f_j)(x)\right)^{\eta}+M \left(T(\vec{f})^\eta\right)(x)+\frac{1}{|B|}\int_{B}|T(\vec{f_0})(z)|^\eta dz.\\
 \end{eqnarray*}
 The left part is the same as in \cite{duonggong}. Then we finish the proof of this Lemma.

\end{proof}

\begin{lemma}\label{32}(\cite{Duoandikoetxea}) If $\omega\in A_p$ and $p> 1$, then $M$ maps from $L^{p }(\omega)$ to $L^{p }(\omega)$.
\end{lemma}
\begin{lemma}\label{33}(\cite{Chen}) If $\omega\in A_p$ and $p\geq 1$, then $M$ maps from $L^{p,\infty}(\omega)$ to $L^{p,\infty}(\omega)$.
\end{lemma}
\textbf{Proof of Theorem \ref{lemma2}.}

\begin{proof} We choose a positive number $\eta$ such that $\eta< p$, then
Theorem \ref{lemma2}(ii) follows by using Lemma \ref{31}, Lemma \ref{33}, Theorem \ref{thm:22} and the H\"{o}lder inequality for weak spaces(see \cite{grafkosbook}, p.15).
\begin{eqnarray*}
||T^*(\vec{f})||_{L^{p,\infty}(\omega)}&\leq &C \left( ||M_\eta(T(\vec{f}))||_{L^{p,\infty}(\omega)}+||\prod_{j=1}^m M(f_j)||_{L^{p,\infty}(\omega)}\right)\\
&= &C \left( ||M(|T(\vec{f}|^\eta))||^{\frac{1}{\eta}}_{L^{\frac{p}{\eta},\infty}(\omega)}+\prod_{j=1}^m|| M(f_j)||_{L^{p_j,\infty}(\omega)}\right)\\
&\leq &C \left( |||T(\vec{f})|^\eta||^{\frac{1}{\eta}}_{L^{\frac{p}{\eta},\infty}(\omega)}+\prod_{j=1}^m ||f_j||_{L^{p_j}(\omega)}\right)\\
&\leq &C \left( ||T(\vec{f})||_{L^{p,\infty}(\omega)}+\prod_{j=1}^m ||f_j||_{L^{p_j}(\omega)}\right)\\
&\leq & C \prod_{j=1}^m ||f_j||_{L^{p_j}(\omega)}.
\end{eqnarray*}
The proof of (i) is similar.
\end{proof}

\end{document}